\documentclass{article}
\usepackage[utf8]{inputenc}
\usepackage{amsmath,amsfonts,amssymb,amsthm}
\usepackage{hyperref}
\usepackage{booktabs} 
\usepackage{enumitem}
\usepackage{geometry}

\makeatletter
\newtheorem*{rep@theorem}{\rep@title}
\newcommand{\newreptheorem}[2]{
\newenvironment{rep#1}[1]{
 \def\rep@title{#2 \ref{##1}}
 \begin{rep@theorem}}
 {\end{rep@theorem}}}
\makeatother

\newtheorem{thm}{Theorem}
\newreptheorem{thm}{Theorem}
\newtheorem{lem}[thm]{Lemma}
\newtheorem{cor}[thm]{Corollary}
\newtheorem{prop}[thm]{Proposition}
\newtheorem{conj}[thm]{Conjecture}

\newtheorem{problem}[thm]{Problem}
\newtheorem*{thm*}{Theorem}

\theoremstyle{definition}
\newtheorem{defn}[thm]{Definition}

\DeclareMathOperator{\cyc}{cyc}
\DeclareMathOperator{\ord}{ord}
\DeclareMathOperator{\lcm}{lcm}
\DeclareMathOperator{\fix}{fix}

\newcommand{\Z}{\mathbb{Z}}
\newcommand{\F}{\mathbb{F}}

\newcommand{\Id}{\mathrm{Id}}

\usepackage[noblocks]{authblk} 

\newcommand\emailfont{\sffamily}

\title{Circular sorting, strong complete mappings and wreath product constructions}

\author[1]{Paul Bastide\thanks{Research supported by ERC Advanced Grant 883810.}}
\author[2]{Anurag Bishnoi}
\author[3]{Carla Groenland\thanks{Research supported by the Dutch Research Council (NWO, VI.Veni.232.073).}}
\author[4]{Dion Gijswijt}
\author[5]{Rohinee Joshi}

\affil[1]{
Mathematical Institute, University of Oxford, UK. Email: {\emailfont{paul.bastide@ens-rennes.fr}}.}

\affil[2,3,4,5]{
Delft Institute of Applied Mathematics, TU Delft, the Netherlands. Emails: 
{\emailfont{a.bishnoi@tudelft.nl}}, {\emailfont{c.e.groenland@tudelft.nl}}, {\emailfont{d.c.gijswijt@tudelft.nl}} and {\emailfont{rohi1729@gmail.com}}.
}

\date{\today}

\begin{document}
\maketitle
\begin{abstract}
    We continue the study of Adin, Alon and Roichman [\textit{arXiv:2502.14398}, 2025] on the number of steps required to sort $n$ labelled points on a circle by transpositions. Imagine that the vertices of a cycle of length $n$ are labelled by the elements $1,\dots,n$. We are allowed to change this labelling by swapping the labels of any two vertices on the cycle. How many swaps are needed to obtain a labelling that has the elements $1,\dots,n$ in clockwise order?

    We provide evidence for their conjecture that at most $n-3$ transpositions are needed to sort a circular permutation when $n$ is not prime. We prove this conjecture when $2\mid n$ or $3\mid n$ or when restricting to permutations given by a polynomial over $\mathbb{Z}_n$. We also provide various algebraic constructions of circular permutations that require many transpositions to sort, most notably providing one that matches our upper bound when $n=3p$ for $p$ an odd prime, and disproving their second conjecture by providing non-affine circular permutations that require $n-2$ transpositions (for $n$ prime). 
    We also improve the lower bounds for some sequences of composite numbers. 
    Finally, we improve the bounds for small $n$ computationally. In particular, we prove a tight upper bound for $n=25$ via an exhaustive computer search using a new connection between this problem and strong complete mappings. 
\end{abstract}

\section{Introduction}
The following problem on permutations was introduced by Adin, Alon and Roichman in \cite{adin2025circularsorting}. Suppose we are given $n$ points arranged on a circle that are bijectively labelled with elements from $[n]=\{1,2,\ldots, n\}$. Consider the smallest number of swaps needed to sort them into clockwise order. Taking the worst case over all labellings, what is the maximum number of swaps needed to sort $n$ labels into clockwise order?

More formally, a \emph{circular permutation} is a left coset $[\pi] \in S_n / \langle c \rangle$, where $c:= (1\,2\,\dots\,n)$ denotes the cyclic shift. For $\pi\in S_n$ we define $t(\pi)$ to be the minimum number of factors in a decomposition of $\pi$ as a product of transpositions. Note that $t(\pi)=n-\cyc(\pi)$, where $\cyc(\pi)$ is the number of cycles (including fixed points) in the cycle decomposition of $\pi$. We are now interested in determining the \emph{circular sorting number} $t(n)$ given by
\[
t(n):=\max_{\pi\in S_n} t([\pi]),\quad \text{where } t([\pi]):=\min_{\sigma\in[\pi]}t(\sigma).
\]
Note that since $\pi c^k$ and $c^k\pi$ have the same cycle type, defining $[\pi]$ to be a right coset instead of a left coset would not change the value of $t([\pi])$ and $t(n)$.

\paragraph{Example.} 
Take $n=6$ and let $\pi=(1\,2\,4\,3\,6)(5)$. Then 
\[
[\pi]=\{(1\,2\,4\,3\,6)(5),\ (1\,4\,5)(2\,6)(3),\ (1\,6\,4)(2\,3\,5),\ (1\,3)(2\,5\,4)(6),\ (1\,5\,6\,3\,2)(4),\ (1)(2)(3\,4\,6\,5)\}.
\]
So we find $t([\pi])=6-3=3$ and hence $t(6)\geq 3$. By exhaustive search over all permutations in $S_6$ or using the fact that $\Z_6$ has no orthomorphism (see Section~\ref{sec:non_prime}), it follows that $t(6)=3$.

\paragraph{}
It will often be convenient to replace the label set $\{1,2,\ldots, n\}$ by other sets of size $n$,  having additional structure. Throughout the paper we will mostly use the cyclic group $\Z_n=\Z/n\Z$ as label set and thus consider permutations as elements of the set $S(\Z_n):=\{f:\Z_n\to\Z_n:\text {$f$ is a bijection}\}$. With this convention, the elements of $[\pi]$ are precisely the maps $x\mapsto \pi(x+b)$ where $b\in \Z_n$. 

\paragraph{Related sorting problems.}
Distances in Cayley graphs of the symmetric group $S_n$ define many classical sorting problems on permutations. Taking adjacent transpositions as generator set, the Cayley distance equals the inversion number, while the set of all transpositions give distance $n-\cyc(\pi)$ and diameter $n-1$. Well-known `recreational' examples concern sorting a bridge hand~\cite{eriksson2001sortingbridge}, flipping a stack of pancakes~\cite{Bulteau15pancakehard,Cibulka11pancake,gyori1978pancakes} and the diameter of the `Rubik's cube graph'~\cite{rokicki2013diameter}.  

In computational biology, cyclically adjacent transpositions $(1\,2),(2\,3),\ldots,(n\,1)$ are used to model rearrangements on a circular chromosome; it is known that at most $n^{2}/4$ moves are required and that optimal sequences can be found in polynomial time~\cite{FCS2010catslowerbound,jerrum1985complexity,vZBSY2016catsupperbound}. Other operations studied in this context are swapping adjacent segments~\cite{Elias05sortingbytranspositions} or reversing the order inside a segment~\cite{bafna1996genomereversals,Christie98reversals,gatesPapadimitriou1979prefix}. 

Although these problems consider generator sets that are invariant under cyclic shifts, the permutations themselves are not considered up to cyclic shifts. In contrast, the circular sorting problem we consider here is equivalent to determining the diameter of the \emph{Schreier coset graph} of $S_n$ with respect to the subgroup $\langle c\rangle$ and the generating set consisting of all transpositions.

\paragraph{Previous results.}
Given $\pi\in S(\Z_n)$, every $x\in \Z_n$ is a fixed point of exactly one of the permutations in the coset $[\pi]$. It follows that for $n\geq 2$ we have $t(n)\leq n-2$. In~\cite{adin2025circularsorting}, the authors show that equality holds when $n$ is a prime number by considering \emph{affine permutations}. These are permutations $\pi:\Z_n\to \Z_n$ of the form $\pi(x)=ax+b$ where $a\in \Z_n^{\times}$ and $b\in \Z_n$. When $n$ is prime and $a$ is a generator of $\Z_n^\times$, the permutation $\pi$ satisfies $t([\pi])=n-2$. They conjectured that the upper bound $t(n)\leq n-2$ is an equality only if $n$ is prime.

\begin{conj}[{\cite[Conj.~3.4]{adin2025circularsorting}}]\label{conj:non-prime}
If $n$ is composite, then $t(n) \leq n - 3$.
\end{conj}

Moreover, they conjectured that for $n$ prime, equality is only obtained by affine permutations.
\begin{conj}[{\cite[Conj.~3.6]{adin2025circularsorting}}]\label{conj:affine_prime}
If $\pi \in S(\Z_n)$ satisfies $t([\pi]) = n-2$, then $n$ is prime and $\pi$ is affine.
\end{conj}

For general $n\geq 2$ and odd primes $p$, they proved the following lower bounds by analysing the cycle structure of affine maps in $S(\Z_n)$:
\[
t(p^k) \geq p^k - k - 1, \qquad
t(2p^k) \geq 2p^k - 2k - 2, \qquad
t(2^k) \geq 2^k - 2k + 1.
\]
Using probabilistic arguments they also derived the general lower bound $t(n) \ge n - e(\ln n + 1)$.

\paragraph{Results and paper outline.}
In Section~\ref{sec:non_prime}, we establish a relation with orthomorphisms and strong complete mappings and provide evidence for Conjecture~\ref{conj:non-prime}.
\begin{thm}\label{thm:div_by_2_or_3}
  Let $n\geq 4$. Then $t(n)\leq n-3$ if $n$ is divisible by $2$ or $3$.
\end{thm} 

We also  disprove Conjecture~\ref{conj:affine_prime} by constructing non-affine permutations $\pi$ with $t([\pi]) = n-2$ for several primes, the smallest being $n=23$. Our counterexamples arise from quadratic orthomorphisms~\cite{evans1987orthomorphisms} and we find that for $n=23$ all counterexamples are of this form (see Proposition~\ref{prop:quadratic}).

In Section~\ref{sec:affine}, we analyse the cycle structure of affine maps $\pi\in S(\Z_n)$ and
show that the maximum value of $t([\pi])$ for such affine permutations is $n-\sum_{d\mid n}\frac{\varphi(d)}{\lambda(d)}$, for $\lambda(d)$ the maximum order of an element of $\mathbb{Z}_d^\times$  and $\varphi(d)= |\Z_d^\times |$ (see Theorem~\ref{thm:affine_general}). 

In Section~\ref{sec:semi-direct}, we describe a construction of elements in the wreath product $S(\Z_n)\wr S(\Z_m)$ that yield new lower bounds:
\begin{thm}\label{thm:lower_bounds}
\begin{itemize}
\item[\textup{(i)}] Let $n,m\geq 2$ be integers. Then $t(nm)\geq nt(m)+t(n)$.
\item[\textup{(ii)}] Let $p,q\geq 3$ be primes. Then $t(pq)\geq pq-5$ and $t(pq)\geq pq-3$ if moreover $q-1\mid p-1$.
\end{itemize}
\end{thm}
In the special case $q=3$, combined with Theorem~\ref{thm:div_by_2_or_3} this gives the following tight result.
\begin{cor}\label{cor:3p-3}
    For $p$ prime, $t(3p)=3p-3$.
\end{cor}

In Section~\ref{sec:permpolys}, we provide further evidence towards Conjecture~\ref{conj:non-prime}.
\begin{prop}
\label{prop:permpoly}
   $t([\pi])\leq n-3$ when $n$ is not prime and $\pi\in S_n$ can be expressed as a permutation polynomial with coefficients in $\Z_n$. 
\end{prop}

In Section~\ref{sec:computational}, we explain the computational approach that we used to provide new bounds on $t(n)$ for small $n$, including a computational proof of $t(25)=22$ that uses the link with strong complete mappings.
Here, we also discuss several other interesting examples of permutations that we found.

Finally, in Section~\ref{sec:conclusion} we discuss several remaining open problems.

\section{Relation with strong complete mappings}
\label{sec:non_prime}
The permutations $\pi\in S(\Z_n)$ that satisfy $t([\pi])=n-2$ are of particular interest for the circular sorting problem.  In this section we show that these permutations are strong complete mappings. We use this to show that $t(n)\leq n-3$ if $n$ is divisible by $2$ or $3$ (Theorem~\ref{thm:div_by_2_or_3}) and to obtain counterexamples to Conjecture~\ref{conj:affine_prime}. We first provide the relevant definitions. 
\begin{defn}
    Let $n$ be a positive integer and denote by $\Id\in S(\Z_n)$ the identity permutation. A permutation $\pi\in S(\Z_n)$ is called 
\begin{itemize}
    \item an \emph{orthomorphism} if $\pi-\Id$ is a permutation;
    \item a \emph{complete mapping} if $\pi+\Id$ is a permutation;
    \item a \emph{strong complete mapping} if both $\pi-\Id$ and $\pi+\Id$ are permutations.
\end{itemize}
\end{defn}
Note that if $\pi\in S(\Z_n)$ is an orthomorphism (respectively complete mapping), then so are the cyclic shifts $\pi\circ c$ and $c\circ \pi$ of $\pi$. 

It is well known in the literature of orthomorphisms that an orthomorphism $\pi\in S(\Z_n)$ exists if and only if $n$ is even. This fact is often attributed to Euler~\cite{euler1782recherches} who proved this in the language of orthogonal Latin squares. We refer the reader to the excellent book by Evans~\cite{evans2018orthogonallatinsquares} for background information on orthomorphisms and their link with orthogonal Latin squares.

A strong complete mapping $\pi\in S(\Z_n)$ corresponds precisely to an arrangement of $n$ non-attacking queens on a toroidal $n\times n$ chessboard. That such an arrangement exists if and only if $\gcd(n,6)=1$, was shown by P\'olya~\cite{Polya} in 1918 (attributing the proof to Hurwitz).
\begin{prop}[P\'olya 1918]\label{Polya}
Let $n$ be a positive integer. There is a strong complete mapping $\pi\in S(\Z_n)$ if and only if $\gcd(n,6)=1$.
\end{prop}
We included a proof of this fact in the appendix for completeness.

To establish the connection between strong complete mappings and the circular sorting problem, we recall the following lemma from~\cite{adin2025circularsorting}.
\begin{lem}[Lemma 3.5 in~\cite{adin2025circularsorting}]
\label{lem:n-2}
Let $n\geq 3$ be an integer and let $\pi\in S_n$. Then $t([\pi])=n-2$ if and only if each cyclic shift of $\pi$ has cycle type $(1,n-1)$. 
\end{lem}
\begin{proof}
The `if' part follows directly. 
For the `only if' part, suppose that $t([\pi])=n-2$. Every $x\in \Z_n$ is a fixed point of exactly one cyclic shift $\pi\circ c^k$ since $\pi(x+k)=x$ has exactly one solution $k\in \Z_n$. Since $\cyc(\pi\circ c^k)\leq 2$ for every $k$, no cyclic shift can have more than one fixed point. Hence, every cyclic shift has exactly one fixed point and therefore also one cycle of length $n-1$. 
\end{proof}

\begin{lem}
\label{lem:orthomorphism_doublefixed}
Let $\pi\in S(\Z_n)$. Then $\pi$ is an orthomorphism if and only if every cyclic shift $\pi\circ c^k$ has at most one fixed point.
\end{lem}
\begin{proof}
Suppose that some cyclic shift $\pi\circ c^k$ has two distinct fixed points $x,y$. Then $\pi(x+k)=x$ and $\pi(y+k)=y$, so $(\pi-\Id)(x+k)=-k=(\pi-\Id)(y+k)$. It follows that $\pi-\Id$ is not injective, so $\pi$ is not an orthomorphism.

Now suppose that $\pi$ is not an orthomorphism. Then $\pi-\Id$ is not injective, say $(\pi-\Id)(x)=k=(\pi-\Id)(y)$ for distinct $x,y$. Setting $a=x+k$ and $b=y+k$ we find $\pi(a-k)=a$ and $\pi(b-k)=b$. It follows that $\pi\circ c^{-k}$ has two distinct fixed points.
\end{proof}
\begin{lem}
\label{lem:completemapping_transposition}
Let $\pi\in S(\Z_n)$. Then $\pi$ is a complete mapping if and only if no cyclic shift of $\pi$ has a transposition in its cycle decomposition.
\end{lem}
\begin{proof}
Suppose that $\pi\circ c^k$ has a transposition $(x\, y)$ in its cycle decomposition. So $\pi(x+k)=y$ and $\pi(y+k)=x$. Then $(\pi+\Id)(x+k)=(\pi+\Id)(y+k)$, so $\pi+\Id$ is not injective and therefore $\pi$ is not a complete map.

Suppose that $\pi$ is not a complete map. Then $\pi+\Id$ is not injective. Hence there exist distinct $x,y$ and $s$ so that $(\pi+\Id)(x)=s=(\pi+\Id)(y)$. Set $a:=s-y$, $b:=s-x$ and $k:=x+y-s$. Then $\pi(a+k)=s-x=b$ and $\pi(b+k)=s-y=a$. So $\pi\circ c^k$ has transposition $(a\, b)$ in its cycle decomposition.
\end{proof}

Lemma~\ref{lem:n-2} and the two lemmas above immediately imply the following result.
\begin{prop}
\label{prop:relation_to_strong_complete}
Let $n\geq 4$ and let $\pi\in S(\Z_n)$. If $t([\pi])=n-2$, then $\pi$ is a strong complete mapping.
\end{prop}
\begin{proof}
By Lemma~\ref{lem:n-2}, every cyclic shift of $\pi$ has cycle type $(1,n-1)$. Since $n-1\geq 3$ it follows that no cyclic shift of $\pi$ has more than one fixed point or a transposition in its cycle decomposition. Hence, by Lemma~\ref{lem:orthomorphism_doublefixed} and Lemma~\ref{lem:completemapping_transposition}, $\pi$ is a strong complete mapping. 
\end{proof}

The proof of Theorem~\ref{thm:div_by_2_or_3} (which we restate below) now follows directly from Proposition~\ref{prop:relation_to_strong_complete}.
\begin{repthm}{thm:div_by_2_or_3}
  Let $n\geq 4$. Then $t(n)\leq n-3$ if $n$ is divisible by $2$ or $3$.
\end{repthm}
\begin{proof}
Since $\gcd(n,6)>1$, there is no strong complete mapping $\pi\in S(\Z_n)$ by Proposition~\ref{Polya}. It then follows from Proposition~\ref{prop:relation_to_strong_complete} that $t[\pi]\leq n-3$ for every $\pi\in S(\Z_n)$, which implies that $t(n)\leq n-3$. 
\end{proof}

For $n$ even, another simple proof of $t(n)\leq n-3$ is possible. Indeed, when $n$ is even, the sign of cyclic shift $c$ is odd, and so applying it must change the sign of the permutation and in particular change the cycle type. So no permutation can have cycle type $(1,n-1)$ across all cyclic shifts.

\paragraph{Counterexamples to Conjecture~\ref{conj:affine_prime}}
\label{sec:non_affine}
Using the connection to strong complete mappings, we provide counterexamples to Conjecture~\ref{conj:affine_prime} from
Adin, Alon and Roichman~\cite[Conjecture 3.4]{adin2025circularsorting}. Recall that the conjecture says that if $\pi\in S(\Z_n)$ is a permutation with $t([\pi])=n-2$, then $n$ is prime and $\pi:x \mapsto ax+b$ for some $a\in \Z_n^\times$ and $b\in \Z_n$. Since every permutation has a cyclic shift that has $0$ as a fixed point, we may restrict ourselves to such \emph{normalized} permutations (so $b=0$).

By iterating over the strong complete mappings (see also Section \ref{sec:computational}), we show computationally that the smallest $n$ for which counterexamples exist is $n=23$. We also determine all counterexamples for $n=23$.
For $n=23$ there are $194$ permutations $\pi$ with $\pi(0)=0$ such that all cyclic shifts of $\pi$ have cycle type $(1,n-1)$, whereas only $\phi(22)=10$ of these are affine. 
We include a specific counterexample to the conjecture below: 
\[
\pi=(0,2,7,21,17,12,4,9,14,19,1,8,15,22,6,3,20,13,11,18,5,10,16).
\]
Above, we use the notation $\pi=(\pi(0),\dots,\pi(n-1)).$
All counterexamples that we have found have the following algebraic structure. 
For $p$ a prime number, $a,b\in \Z_p^{\times}$, we define $[a,b]\in S(\Z_p)$ by
\[
[a,b](x) := \begin{cases} 
ax & \text{ if $x = y^2$ for some $y \in \Z_p$},\\
bx & \text{ otherwise}.
\end{cases}
\] 

Orthomorphisms of this form are known as \emph{quadratic orthomorphisms} \cite{evans1987orthomorphisms, mendelsohn1985search} (see \cite[Section 9.3.2]{evans2018orthogonallatinsquares} for a survey).
Note that as polynomials, they have degree $(p + 1)/2$ (unless $a = b$) as 
\[[a, b](x) = \left(\frac{a - b}{2}\right) x^{(p + 1)/2} + \left(\frac{a + b}{2}\right) x.\]

Using an exhaustive computer search over all strong complete maps in $S(\Z_{23})$, described in Section~\ref{sec:computational}, we can prove the following. 
\begin{prop}\label{prop:quadratic}
For $\pi\in S_{23}$, $t([\pi])=21$ if and only if 
$\pi(x)=[a,b](x + c)+d$ for some $c,d \in \Z_{23}$ and 
\[(a, b) \in \{(5, 7), (7, 5), (10, 14), (14, 10), (11, 15), (15, 11), (20, 21), (21,20)\}
\]
or $a=b$ (so $\pi$ is affine) and 
\[
a\in \{5, 7, 10, 11, 14, 15, 17, 19, 20, 21\}.
\]
\end{prop}
There are $184=23\cdot 8$ options for $\pi$ with $\pi(0)=0$ among the non-affine examples. Note that if $\pi$ has the required cycle type for all cyclic shifts and $\pi(0)=0$, then the same holds for $\pi_i: x\mapsto \pi(x+i)-\pi(i)$ for every $i\in \Z_n$. If $\pi$ is not affine (and $n$ is prime) then these are all different. Hence, we get $23\cdot 8=184$ normalised examples.

\medskip

For a prime number $p$, the map $[a,b]\in S(\Z_p)$ is a strong complete mapping if and only if $ab$, $(a-1)(b-1)$, and $(a+1)(b+1)$ are non-zero squares in $\Z_p$ (see \cite{evans2018orthogonallatinsquares}). Searching specifically for quadratic strong complete mappings, we found examples of non-affine permutations $\pi\in S(\Z_p)$ that satisfy $t([\pi])=p-2$ for the following primes $p$ below $500$:
\begin{align*}
 &23, 31, 41, 59, 71, 83, 89, 97, 103, 113, 131, 137, 139, 149, 151, 157, 163, 173, 179, 181, 197, 199,\\ 
 &211, 223, 227, 239, 251, 271, 283, 293,
 307, 311, 331, 347, 349, 367, 379, 383,389, 401, 409, 419,\\ 
 &431,433, 461, 467, 487, 491, 499. 
\end{align*}

\section{Lower bounds from affine permutations}\label{sec:affine}
In this section, we restrict to affine permutations $\pi\in S(\Z_n)$. That is, permutations of the form $\pi(x)=ax$, where $a\in \Z_n^\times$. We define
\[
t_\mathrm{aff}(n):=\max\{t([\pi]): \pi\in S(\Z_n)\text{ is affine}\}.
\]
Clearly, $t_\mathrm{aff}(n)$ is a lower bound on $t(n)$. We will determine the number $t_\mathrm{aff}(n)$ in terms of the prime factorisation of $n$ for all positive $n$. 

We recall that for odd primes $p$ the group $\Z_{p^k}^\times$ is cyclic of order $\varphi(p^k)=(p-1)p^{k-1}$. The group $\Z_{2^k}^\times$ is cyclic for $k\leq 2$ and isomorphic to $\Z_2\times \Z_{2^{k-2}}$ for $k\geq 3$ (see~\cite[Art.~90–91]{Gauss1986Disquisitiones}). The element $3\in \Z_{2^k}^\times$ is an element of maximum order for every $k$ and we denote this order\footnote{In general one defines the Carmichael function $\lambda(n)$ as the maximum order of an element of $\Z_n^\times$} by $\lambda(2^k)$. So $\lambda(2^k)=\varphi(2^k)$ if $k\leq 2$ and $\lambda(2^k)=\varphi(2^k)/2$ for $k\geq 3$.

For an element $a\in \Z_{n}^\times$ and a positive divisor $d\mid n$ we denote by $\ord_d(a)$ the order of $(a \bmod{d})$ in $\Z_d^\times$.
\begin{lem}\label{lem:affine_cycle}
    Let $n$ be a positive integer and let $a\in \Z_n^\times$.
    Then the number of cycles in the cycle decomposition of the permutation $\pi: x \mapsto ax$ in $\Z_n$ is equal to 
    \[\sum_{d \mid n} \frac{\varphi(n/d)}{\ord_{n/d}(a)} = 
    \sum_{d \mid n} \frac{\varphi(d)}{\ord_{d}(a)}.\]
\end{lem}
\begin{proof}For a positive divisor $d\mid n$ denote $V_d:=d\Z_n^\times$. More explicitly,
\begin{align*}
V_d&=\{i \bmod{n}: i\in \{0,\ldots, n-1\},\ \gcd(i,n)=d\}\\
&=\{di \bmod{n}: i\in \{0,\ldots, \tfrac{n}{d}-1\},\ \gcd(i,\tfrac{n}{d})=1\}.
\end{align*}
From the first equality, we see that the sets $V_d$ partition $\Z_n$ and each $V_d$ is a union of orbits of $\pi$. From the second equality, we see that $|V_d|=\varphi(n/d)$. 

For any $x\in V_d$ the length of the cycle of $\pi$ containing $x$ is equal to $\ord_{n/d}(a)$ since 
\[
a^\ell x=x \iff (a^\ell-1)d\equiv 0 \pmod{n} \iff a^\ell-1 \equiv 0\pmod{n/d}.
\]
We find that $V_d$ is the union of $\frac{\varphi(n/d)}{\ord_{n/d}(a)}$ cycles. Summing over all positive divisors $d\mid n$ concludes the proof. 
\end{proof}

The next lemma shows that for affine permutation $\pi$, the normalised cyclic shift determines $t([\pi])$. 
\begin{lem}\label{lem:affine_max}
Let $n$ and $a$ be coprime positive integers. For every integer $k$ define $\pi_k\in S(\Z_n)$ by $\pi_k(x)=ax+k$. Then $\cyc(\pi_k)\leq \cyc(\pi_0)$.
\end{lem}
We remark that if $a-1\in \Z_n^\times$, then in fact $\cyc(\pi_k)=\cyc(\pi_0)$. Indeed, taking $u$ such that $(a-1)u\equiv k\pmod{n}$ we have  $\pi_k(x)=ax+k=a(x+u)-u=\pi(x+u)-u$, so $\pi_k$ and $\pi_0$ are conjugated permutation and therefore have the same cycle type. We now prove the general case. 
\begin{proof}
Let $m$ be the order of $a$ in $\Z_n^\times$. For positive $\ell\in \Z$ we have
\[
\pi_k^\ell(x)=a^\ell x+k(a^{\ell-1}+\cdots+a+1).
\]
It follows that $\pi_k^\ell$ is the identity map if and only if $k(a^{\ell-1}+\cdots+a+1)=0$ and $a^\ell=1$. In particular, $\pi_0$ has order $m$ and the order of $\pi_k$ is a positive multiple of $m$. 

Let $\pi_k$ have order $tm$ for $t\geq 1$. Using Burnside's lemma, we have 
\[
\cyc(\pi_k)=\frac{1}{tm}\sum_{\ell=0}^{tm-1} \fix(\pi_k^\ell),
\]
where $\fix(\pi)$ denotes the number of fixed points of a permutation $\pi$. The number of fixed points of $\pi_k^\ell$ is equal to the number of solutions $x\in \Z_n$ to $(a^\ell-1)x=-k(a^{\ell-1}+\cdots+a+1)$. Hence
\[
\fix(\pi_k^\ell)=\begin{cases}0&\text{if $k(a^{\ell-1}+\cdots+a+1)\not\in (a^\ell-1)\Z_n$}\\
\gcd(n,a^{\ell}-1)&\text{otherwise.}\end{cases}
\]
So $\fix(\pi_k^\ell)\leq \fix(\pi_0^\ell)$. Since $\pi_0^m=\Id$ we find
\[
\cyc(\pi_k)=\frac{1}{tm}\sum_{\ell=0}^{tm-1} \fix(\pi_k^\ell)\leq \frac{1}{tm}\sum_{\ell=0}^{tm-1} \fix(\pi_0^\ell)=\frac{1}{m}\sum_{\ell=0}^{m-1} \fix(\pi_0^\ell)=\cyc(\pi_0).\qedhere
\]
\end{proof}
Combining Lemma~\ref{lem:affine_cycle} and Lemma~\ref{lem:affine_max} we immediately obtain the following result.

\begin{prop}\label{prop:affine}
    Let $a, n$ be coprime positive integers. Let $\pi\in S(\Z_n)$ be the affine permutation given by $\pi(x)=ax$. Then 
    \[
    t([\pi])=n-\sum_{d \mid n} \frac{\varphi(d)}{\ord_{d}(a)}.
    \]
\end{prop}

By expressing the maximum of $t([\pi])$ over all affine $\pi$ in terms of the prime factorisation of $n$, we obtain the following expression for the affine circular sorting number. 
\begin{thm}\label{thm:affine_general}
    Let $n$ be a positive integer with prime factorisation $n = 2^{k_0}p_1^{k_1} \cdots p_r^{k_r}$. Then 
\[
t_\mathrm{aff}(n)=n-\sum_{\substack{0\leq m_i\leq k_i\\ i=0,\ldots, r}} \frac{\varphi(2^{m_0})\cdot \varphi(p_1^{m_1})\cdots\varphi(p_r^{m_r})}{\lcm\{\lambda(2^{m_0}),\varphi(p_1^{m_1}),\ldots, \varphi(p_r^{m_r})\}}
.\]
\end{thm}
\begin{proof}
Let $a$ be a positive integer coprime to $n$ and define $\pi\in S(\Z_n)$ by $\pi(x):=ax$. Let $d=2^{m_0}p_1^{m_1}\cdots p_r^{m_r}$ be a divisor of $n$. Then 
\[
\ord_d(a)=\lcm \{\ord_{2^{m_0}}(a),\ord_{p_1^{m_1}}(a),\ldots, \ord_{p_r^{m_r}}(a)\}.
\]
By the Chinese remainder theorem, we can choose $a$ such that $a\equiv 3\pmod{2^{m_0}}$ and $a$ generates $\Z_{p_i}^{k_i}$ for $i\in [r]$. With this choice, we find 
\[
\ord_d(a)=\lcm \{\lambda(2^{m_0}),\varphi(p_1^{m_1}),\ldots, \varphi(p_r^{m_r})\},
\]
and it is clear that for this choice of $a$ all values $\ord_d(a)$ are maximised. By Proposition~\ref{prop:affine} it follows that 
\[
t_\mathrm{aff}(n)=t([\pi])=n-\sum_{\substack{0\leq m_i\leq k_i\\ i=0,\ldots, r}} \frac{\varphi(2^{m_0})\cdot \varphi(p_1^{m_1})\cdots\varphi(p_r^{m_r})}{\lcm\{\lambda(2^{m_0}),\varphi(p_1^{m_1}),\ldots, \varphi(p_r^{m_r})\}}.\qedhere
\]
\end{proof}
We remark that $t_\mathrm{aff}$ can also be written in terms of the Carmichael function $\lambda$, where $\lambda(n)$ is the exponent of $\Z_n^\times$,  as $t_\mathrm{aff}(n)=n-\sum_{d\mid n}\frac{\varphi(d)}{\lambda(d)}$. 

As special cases of Theorem~\ref{thm:affine_general}, we recover the following bounds from~\cite{adin2025circularsorting}, 
where $p$ denotes an odd prime and $k\geq 1$ an integer:
\[
  t_{\mathrm{aff}}(p^k)= p^k-k-1,\qquad
  t_{\mathrm{aff}}(2p^k)= 2p^k-2k-2,\qquad
  t_{\mathrm{aff}}(2^{k+1})= 2^{k+1}-2k-1.
\]

\section{Wreath product construction}
\label{sec:semi-direct}
Let $m,n$ be positive integers. Let $G$ be the subgroup of $S(\Z_m\times \Z_n)$ consisting of all $\sigma$ of the form 
\begin{equation}\label{eq:semidirect}
    \sigma(x,y)=(\pi(x), \pi_{x}(y))
\end{equation}
where $\pi\in S(\Z_m)$ and $\pi_x\in S(\Z_n)$ for all $x\in\Z_m$. So $G$ is isomorphic to the wreath product $S(\Z_n)\wr S(\Z_m)$. If we think of $G$ as acting on a grid with $m$ columns and $n$ rows, the elements of~$G$ first permute the elements in each column~$x$ using $\pi_x$ and then permute the columns using $\pi$. When working in a group $\Z_n$, we will often write `$k$' instead of `$k\bmod{n}$' for integers $k$ to avoid cumbersome notation.

Let $c\in G$ be given by 
\begin{equation}\label{eq:c_semidirect}
c(x,y):=\begin{cases}(x+1,y)&\text{if $x\neq m-1$},\\(x+1,y+1)&\text{if $x=m-1$}.\end{cases}
\end{equation}
Note that $c$ is a cyclic permutation of order $nm$.  Identifying the sets $\Z_{mn}$ and $\Z_m\times \Z_n$ via the bijection $r+mt\mapsto (r,t)$ for $r=0,\ldots, m-1$ and $t=0,\ldots, n-1$, the cyclic shift $z\mapsto z+1$ in $S(\Z_{nm})$ corresponds to $c$. 

In this section, we construct permutations $\sigma\in G$ such that $\sigma c^k$ has few cycles for every $k$. We first give an expression for the cyclic shifts $\sigma c^k$ of $\sigma$.
Given an integer $k=r+mt$ where $r\in \{0,\ldots, m-1\}$ and $t\in \Z$, we define 
\[
w_k:\Z_m\to \Z_n,\qquad 
w_k(x):=\begin{cases}t&\text{if } x\in\{0,1,\ldots, m-1-r\},\\
t+1&\text{otherwise}.
\end{cases}
\]
Observe that $c^k(x,y)=(x+k,y+w_k(x))$. More generally, we immediately obtain the following lemma.
\begin{lem}\label{lem:sigmarotate}
Let $\sigma$ be as in (\ref{eq:semidirect}) and let $k$ be an integer. Then, for every $(x,y)\in \Z_m\times \Z_n$ we have
\[
\sigma c^k(x,y)=(\pi(x+k), \pi_{x+k}(y+w_k(x))).
\]
\end{lem}
When repeatedly applying the map $\sigma$ to some starting point $(x,y)$, the first coordinate will perform a cycle according to $\pi$ and if $x$ is in a cycle of length $t$ in $\pi$, then $(x,y)$ will be in a cycle of a length divisible by $t$ in $\sigma$. The following lemma helps with determining which multiple of $t$ this is.

\begin{lem}\label{lem:lift}Let $\sigma$ be as in (\ref{eq:semidirect}) and let $(a_1\, a_2\ldots a_t)$ be a cycle of $\pi$. Then for all $y\in \Z_n$
\[
\sigma^t(a_1,y)=(a_1,\pi'(y)),
\]
where $\pi':=\pi_{a_t}\cdots\pi_{a_2}\pi_{a_1}$. Let the cycles in the cycle decomposition of $\pi'$ have lengths $\ell_1\leq \ell_2\leq\cdots\leq \ell_d$. Then $\sigma
$ over $\{a_1,\ldots, a_t\}\times \Z_n$ induces the disjoint union of $d$ cycles of lengths $t\ell_1,\ldots,t\ell_d$, respectively.
\end{lem}
\begin{proof}
Let $(y_1\, y_2\ldots y_\ell)$ be a cycle of $\pi'$. Then 
\[
\begin{alignedat}{10}
\sigma:\;& (a_1,y_1) &\;  \mapsto\;  & (a_2,\pi_{a_1}y_1) &\; \mapsto\; & \cdots &\; \mapsto\; &(a_t,\pi_{a_{t-1}}\cdots\pi_{a_2}\pi_{a_1}y_1) & \;\mapsto\;\\
& (a_1,y_2) &\; \mapsto\; & (a_2,\pi_{a_1}y_2) &\; \mapsto\; & \cdots &\; \mapsto\; & (a_t,\pi_{a_{t-1}}\cdots \pi_{a_2}\pi_{a_1}y_2) &\; \mapsto\;\\
&&&&&\vdots&&\\
& (a_1,y_\ell) &\; \mapsto\; & (a_2,\pi_{a_1}y_\ell) &\; \mapsto\; & \cdots &\; \mapsto\; & (a_t,\pi_{a_{t-1}}\cdots\pi_{a_2}\pi_{a_1}y_\ell) &\; \mapsto\; & (a_1,y_1)
\end{alignedat}
\]
gives a cycle of length $t\ell$. The resulting $d$ cycles of $\sigma$ are disjoint, contained in $\{a_1,\ldots, a_t\}\times \Z_n$ and their lengths sum to $t(\ell_1+\cdots+\ell_d)=tn=|\{a_1,\ldots, a_t\}\times \Z_n|$, completing the proof.
\end{proof}
The following proposition relates the circular sorting number of a product to that of its factors. We will shortly strengthen this result, but it is instructive first to consider the following simpler case.
\begin{prop}\label{prop:mnbound}
For integers $n,m\geq 2$, we have $t(nm)\geq nt(m)$.
\end{prop}
\begin{proof} Let $\pi'\in S(\Z_m)$ be such that $x\mapsto \pi'(x+k)$ has at most $m-t(m)$ cycles for every integer $k$, and define $\sigma'(x,y):=(\pi'(x),y)$. 
Let $k$ be an integer and set $\sigma:=\sigma'c^k$. It suffices to show that $\sigma$ has at most $n(m-t(m))$ cycles. 

Let $\pi : x \mapsto \pi'(x+k)$. Then $\sigma(x,y)=(\pi(x), y+w_k(x))$. Since $\pi$ has at most $m-t(m)$ cycles, and each corresponds to at most $n$ cycles of $\sigma$ by Lemma~\ref{lem:lift}, it follows that the number of cycles of $\sigma$ is at most $n(m-t(m))$.  
\end{proof}
The construction used above can be seen as taking  $\pi_x'=\Id$ for all $x\in \Z_m$; in the next construction, we choose a different permutation for $\pi_0'$ which will allow us to slightly improve the lower bound.
\begin{prop}
For integers $n,m\geq 2$, we have $t(nm)\geq nt(m)+t(n)$.
\end{prop}
\begin{proof}
Let $\pi'\in S(\Z_m)$ be such that $x\mapsto \pi'(x+k)$ has at most $m-t(m)$ cycles for every integer~$k$, and let $\pi'_{0}\in S(\Z_n)$ be such that $y\mapsto \pi'_{0}(y+k)$ has at most $n-t(n)$ cycles for every integer~$k$. Define $\sigma'\in G$ by
\[
\sigma'(x,y):=(\pi'(x),\pi'_x(y)),\text{ where $\pi'_x=\Id$ for $x\neq 0$}.
\]
Let $k$ be an integer and let $\sigma=\sigma'c^k$. It suffices to show that $\sigma$ has at most $nm-(nt(m)+t(n))$ cycles. 

Observe that we have $\sigma(x,y)=(\pi(x), \pi_x(y))$, where 
\[
\pi(x)=\pi'(x+k),\qquad \pi_x(y)=y+w_k(x)\text{ for $x\neq -k$}, \quad \pi_{-k}(y)=\pi'_{0}(y+w_k(-k))
\]
by Lemma~\ref{lem:sigmarotate}.

Let $(a_1\, a_2\ldots, a_t)$ be a cycle of $\pi$. By Lemma~\ref{lem:lift}, the set $\{a_1,\ldots, a_t\}\times \Z_n$ partitions into at most $n$ cycles of $\sigma$. Now consider in more detail the case that $(a_1\, a_2\ldots, a_t)$ contains the element $-k$. Without loss of generality we may assume that $a_t=-k$. Then $\pi_{a_t}\cdots \pi_{a_1}(y)=\pi_{0}'(y+w_k(a_1)+\cdots+w_k(a_t))$. Since the permutation $y\mapsto \pi_{0}'(y+r)$ has at most $n-t(n)$ cycles for any $r$, we find that by Lemma~\ref{lem:lift} the set $\{a_1,\ldots, a_t\}\times \Z_n$ partitions into at most $n-t(n)$ cycles of $\sigma$. In total, we find that $\sigma$ has at most $n(m-t(m)-1)+(n-t(n))=nm-(nt(m)+t(n))$ cycles. This proves that $t(nm)\geq t([\sigma'])\geq nm-(nm-nt(m)-t(n))=nt(m)+t(n)$ as desired.
\end{proof}
When $(a_1~a_2~\dots~a_t)$ is a cycle of $\pi$, the map $\pi'=\pi_{a_t}\cdots \pi_{a_2}\pi_{a_1}$ from Lemma~\ref{lem:lift} not only depends on the values $a_1,\dots,a_t$ but it can also depend on their order in the cycle. To make analysis easier (or even possible), it is therefore helpful to choose permutations $\pi_x$ that commute. 
When $n$ is a prime number, we know that there are permutations of the form $\pi_a:x\mapsto ax$ with $t([\pi_a])=n-2$, and conveniently all such permutations commute. We exploit this below to obtain better lower bounds in the case that $n$ and $m$ are odd prime numbers.
\begin{lem}\label{lem:generalpq}
Let $p,q\geq 3$ be prime numbers. For each $x\in \Z_p$, let $e_x$ be an integer, and set $e:=\sum_{x\in \Z_p}e_x$. Then
\[
t(pq)\geq pq-\max_x (2+\gcd(q-1,e_x)+\gcd(q-1, e-e_x)).
\]
\end{lem}

\begin{proof}Let $\pi'\in S(\Z_p)$ be such that for every $k$ the cyclic shift $x\mapsto \pi'(x+k)$ has two cycles: one of length~$1$ and one of length~$p-1$. Let $g$ be a generator of $\Z_q^\times$ and define $\pi'_x(y)=g^{e_x}y$. We define $\sigma'\in G$ by $\sigma'(x,y)=(\pi'(x),\pi'_x(y))$. Let $k$ be an integer and let $\sigma=\sigma'c^k$. It suffices to show that $\sigma$ has at most $2+\gcd(q-1,e_x)+\gcd(q-1, e-e_x)$ cycles for some $x\in \Z_p$.

Let $\pi(x)=\pi'(x+k)$. Then 
\[
\sigma(x,y)=(\pi(x),\pi_x(y)),\text{ where $\pi_x(y)=g^{e_{x+k}}\cdot(y+w_k(x))$}.
\]
Let $(a_0)(a_1\, a_2\ldots\, a_{p-1})$ be the cycle decomposition of $\pi$ and write 
$m_i=e_{a_i+k}$ for $i=0,\ldots, p-1$. Define 
\[
b=g^{m_0},\qquad b'=g^{m_1+\cdots+m_{p-1}}=g^{e-m_0}.
\]
Then there are $t,t'\in \Z_q$ such that 
\[
\sigma(a_0,y)=(a_0,by+t),\qquad \sigma^{p-1}(a_1,y)=(a_1, b'y+t').
\]
Since an affine permutation $y\mapsto by+t$ of $\Z_q$ has at most $1+(q-1)/\ord_q(b)$ cycles (equality if $b\neq 1$), it follows by Lemma~\ref{lem:lift} that $\sigma$ has at most $2+\gcd(q-1,m_0)+\gcd(q-1,e-m_0)$
cycles, completing the proof. 
\end{proof}

\begin{prop}
\label{prop:pq-5}
Let $p,q\geq 3$ be prime numbers. Then $t(pq)\geq pq-5$.
\end{prop}
\begin{proof}
We may assume that $p\geq q$. We can then find an integer $a$ such that $0\leq a\leq p$ and $a\equiv 2p-3\pmod{q-1}$. 
Set $e_x:=1$ for $a$ values of $x$ and $e_x:=2$ for the remaining $p-a$ values of $x$. So we have $e=a+2(p-a)\equiv 3\pmod{q-1}$. If $e_x=1$ then $2+\gcd(q-1,e_x)+\gcd(q-1,e-e_x)=2+1+2=5$. If $e_x=2$ then $2+\gcd(q-1,e_x)+\gcd(q-1,e-e_x)=2+2+1=5$. By Lemma~\ref{lem:generalpq}, we conclude that $t(pq)\geq pq-5$.
\end{proof}
Note that this is the best bound we can get from Lemma~\ref{lem:generalpq}. Indeed, if some $e_x$ is even, we get $2+\gcd(q-1,e_x)+\gcd(q-1,e-e_x)\geq 2+2+1=5$. If all $e_x$ are odd, then $e$ is odd because $p$ is odd, hence $e-e_x$ is even and therefore $2+\gcd(q-1,e_x)+\gcd(q-1,e-e_x)\geq 2+1+2=5$.

In the case $q-1\mid p-1$ we can still improve the bound to $t(pq)\geq pq-3$. For this, we need the following lemma.
\begin{lem}\label{lem:circulant}
Let $3\leq q<p$ be prime numbers and assume that $q-1\mid p-1$. Let $f:\Z_p\to \Z_q$ be obtained from a surjective group homomorphism $\Z_p^\times\to \Z_q^\times$ by setting $f(0):=0$. Let $A: \Z_p\times \Z_p\to \Z_q$ be the circulant matrix given by $A_{i+j,i}:=f(j)$. Then $A$ has rank $p-1$.
\end{lem}
\begin{proof}
We identify $\Z_q$ with the field $\F_q$ and let $\F$ be a field extension of $\F_q$ containing a primitive $p$-th root of unity $\zeta$ (for instance $\F=\F_{q^{p-1}}$).  We view $A$ as a matrix over the field $\F$, which does not change its rank. 

Let $Z$ be the Fourier matrix $Z=(\zeta^{ij})_{i,j\in \Z_p}$ with inverse $p^{-1}(\zeta^{-ij})_{i,j\in \Z_p}$. Then $D:=ZAZ^{-1}$ is diagonal with
\[
D_{j,j}=\sum_{i\in \Z_p} f(i)\zeta^{ij}=\sum_{i\in \Z_p^\times} f(i)\zeta^{ij}\text{ for all $j\in \Z_p$.}
\]

Using that $f$ restricts to a surjective group homomorphism $\Z_p^\times\to \Z_q^\times$, we find that  
\[
D_{0,0}=\sum_{i\in \Z_p^\times}f(i)=\frac{p-1}{q-1}\sum_{x\in \Z_q^\times}x=0.
\] 

Let $j\in \Z_p^\times$. Then for every $t\in \Z_p^\times$ we have
\[
D_{j,j}=\sum_{i\in \Z_p^\times} f(i)\zeta^{ij}=\sum_{i\in \Z_p^\times} f(ti)\zeta^{tij}=f(t)\sum_{i\in \Z_p^\times} f(i)\zeta^{itj}=f(t)D_{tj,tj}.
\]
It follows that either $D_{j,j}=0$ for all $j\in \Z_p^\times$ or $D_{j,j}\neq 0$ for all $j\in \Z_p^\times$. The first would imply that $A=0$ which is not the case. So we conclude that all diagonal entries of $D$, except one, are non-zero. So $D$ and $A$ have rank $p-1$.
\end{proof}

\begin{prop} 
\label{prop:q-1dividesp-1}
Let $p,q\geq 3$ be prime numbers with $q-1\mid p-1$. Then $t(pq)\geq pq-3.$
\end{prop}
\begin{proof}
Let $a$ be a generator of $\Z_p^\times$ and define $\pi\in S(\Z_p)$ by $\pi(x):=ax$. Let $g$ be a generator of $\Z_q^\times$. For every $x\in \Z_p$ we define $\pi_x\in S(\Z_q)$ by $\pi_x(y):=gy+b_x$ where the $b_x\in \Z_q$ will be chosen later. We define $\sigma\in G$ by $\sigma(x,y):=(\pi(x),\pi_x(y))$. We will show that we can choose the $b_x$ in such a way that $\sigma c^k$ has $3$ cycles for every $k$. 

Let $r\in \{0,\ldots, p-1\}$ and $t\in \{0,\ldots, q-1\}$ and set $k:=r+tp$. Note that the permutation $x\mapsto \pi(x+r)$ is conjugate to $\pi$. Indeed, setting $d:=ra(1-a)^{-1}\in \Z_p$, we find
\[
\pi(x+r)=ax+ar=ax+(1-a)d=\pi(x-d)+d.
\]
So $x\mapsto \pi(x+r)$ has cycle decomposition $(d)(a^1+d\, a^2+d\,\cdots\, a^{p-1}+d)$.
We have 
\[
\sigma c^k(d,y)=(d, \pi_{d+r}(y+w_k(d))=(d,g(y+w_k(d))+b_{d+r})=(d,gy+(gw_k(d)+b_{d+r})).
\]
Since the permutation $y\mapsto gy+b$ has two cycles for every $b$, the set $\{d\}\times\Z_q$ decomposes into $2$ cycles of $\sigma c^k$ by Lemma~\ref{lem:lift}.

We now consider the cycle $(a^1+d\, a^2+d\cdots a^{p-1}+d)$. We have $(\sigma c^k)^{p-1}(a^1+d,y)=(a^1+d,\pi'(y))$, where $\pi'$ is the composition
\[
\pi'=(y\mapsto g(y+w_k(a^{p-1}+d))+b_{a^{p-1}+d+r})\circ \cdots \circ(y\mapsto g(y+w_k(a^1+d))+b_{a^1+d+r}).
\]
Using that $w_k(x)=t+w_r(x)$, and writing $\beta_i^{(r)}:=b_{a^i+d+r}+gw_r(a^i+d)$ we obtain
\[
\pi'(y)=g^{p-1}y+t(g+\cdots+g^{p-1})+(g^0\beta_{p-1}^{(r)}+\cdots+g^{p-2}\beta_1^{(r)})=y+(g^0\beta_{p-1}^{(r)}+\cdots+g^{p-2}\beta_1^{(r)}).
\]
Here, the second equation follows from the fact that $g^{p-1}=1$ and $g+g^2+\cdots+g^{p-1}=0$ (as $q-1\mid p-1$ and $g$ is a generator of $\Z_q^\times$). If $g^0\beta_{p-1}^{(r)}+\cdots+g^{p-2}\beta_1^{(r)}\neq 0$, then $\pi'$ has only one cycle
and therefore $\{a^1+d,\ldots, a^{p-1}+d\}\times \Z_q$ partitions into only one cycle by Lemma~\ref{lem:lift}. So in that case $\sigma c^k$ has exactly $3$ cycles.  

Write $s:=d+r=r(1-a)^{-1}\in \Z_p$. So we can write
\[
g^0\beta_{p-1}^{(r)}+\cdots+g^{p-2}\beta_1^{(r)} = \sum_{i=1}^{p-1}g^{-i}b_{a^i+s} -u_r,
\]
where $u_r$ is an element in $\Z_q$ that is independent of the values of the $b_x$ that we are about to choose. Since $s=r(1-a)^{-1}$ runs over $\Z_p$ when $r$ runs over $\{0,\ldots, p-1\}$, we can rename $v_s:=u_r$. To complete the proof, it now suffices to show that there exist elements $b_x\in \Z_q$ such that 
\begin{equation}\label{eq:hyperplanes}
\sum_{i=1}^{p-1}g^{-i}b_{a^i+s}=v_s
\end{equation}
is \emph{false} for every $s\in \Z_p$.

Note that if $q=p$ then this follows as the $p$ hyperplanes in $\F_p^p$ corresponding to the equations in (\ref{eq:hyperplanes}) pairwise intersect, which implies that their union has fewer than $p\cdot p^{p-1}=p^p$ elements (i.e. their union is a strict subset of the whole space $\Z_p^p$). So we can assume that $p>q$. 

Observe that the coefficient matrix of the variables $b_x$ is a circulant matrix $A_{s,j+s}:=f(j)$ where $f(0):=0$ and $f(a^i):=g^{-i}$. That is, $\sum_{i=1}^{p-1}g^{-i}b_{a^i+s}=\sum_{j\in \Z_p}A_{s,j+s}b_{j+s}$.
It follows from Lemma~\ref{lem:circulant} that the transpose of $A$ has rank $p-1$. Also note that $A$ has no zero row. So by applying column operations to $A$ and relabeling rows and columns, we can obtain a matrix $A'=\begin{bmatrix}0&\alpha_1,\ldots, \alpha_{p-1}\\0&I_{p-1}\end{bmatrix}\in \Z_q^{p\times p}$, where the $\alpha_1,\ldots, \alpha_{p-1}$ are not all zero. The problem is thus reduced to finding $b'_0,\ldots, b'_{p-1}\in \Z_q$ such that $b'_i\neq v'_i$ for $i=1,\ldots, p-1$ and $\sum_{i=1}^{p-1}\alpha_ib'_i\neq v'_0$, where $v'$ is equal to $v$ up to relabeling the indices.
\end{proof}
Since we showed in Theorem~\ref{thm:div_by_2_or_3} that $t(n)\leq n-3$ when $n$ is divisible by $3$, this proves Corollary~\ref{cor:3p-3}.

When $p,q$ are distinct odd primes, Theorem~\ref{thm:affine_general} shows that $t_{\mathrm{aff}}(pq)=pq-3-\gcd(p-1,q-1)$ and when $p=q$, $t_{\mathrm{aff}}(p^2)=p^2-3$. In particular, if $n=pq$ for $p,q$ odd primes, then our constructions above are at least as good and often better than directly using affine functions. 

\smallskip

The following lemma shows that the construction $\sigma$ studied in this section always has a cyclic shift with at least $3$ cycles and therefore remains in line with Conjecture~\ref{conj:non-prime}.
\begin{lem}
\label{lem:semi_direct_limitation_3cycles}
Let $m,n\geq 2$. 
Let
\[
\sigma(x,y):=(\pi(x),\pi_x(y))
\]
where $\pi\in S(\Z_m)$ and $\pi_x\in S(\Z_n)$ for all $x\in \Z_m$. Let $i\in \Z_m$ be given and let $t$ be an integer such that $y\mapsto \pi_i(y+t)$ has at least $b$ cycles.
Then $\sigma c^k$ has at least $b+1$ cycles for some $k$. 
\end{lem}
\begin{proof}
Let $x_0:=\pi(i)$ and let $r\in \{0,\ldots, m-1\}$ be such that $i-x_0=r$. 

By Lemma~\ref{lem:sigmarotate} there is a $k\in \{r+(t-1)m, r+tm\}$ such that for all $y\in \Z_n$ we have 
\[
\sigma c^{k}(x_0,y)=(\pi(x_0+r),\pi_{x_0+r}(y+t))=(x_0,\pi_i(t)).
\]
This implies, by Lemma~\ref{lem:lift}, that  $\sigma c^{k}$ has at least $b$ cycles on $\{x_0\}\times \Z_n$ and so at least $b+1$ in total.
\end{proof}
In Appendix~\ref{sec:B}, we show that when either $m=2, n\geq 3$ or $m\geq 3, n=2$, every permutation $\sigma$ of the form (\ref{eq:semidirect}) has at least $4$ cycles.

\section{Conjecture~\ref{conj:non-prime} for permutation polynomials}
\label{sec:permpolys}
When proving lower bounds for $t(n)$, one of the first questions that arises is: what is a class of permutations with a nice structure? 
We say that a polynomial $f\in \Z_n[x]$ is a \emph{permutation polynomial} if the function $\pi: \Z_n\to \Z_n$ induced by $f$ is a permutation. In that case, we will say that $\pi$ \emph{admits a permutation polynomial}. We denote the set of $\pi\in S(\Z_n)$ that admit a permutation polynomial by $S_{\text{poly}}(\Z_n)$. 
The goal of this section is to show that Conjecture~\ref{conj:non-prime} holds for permutation polynomials: $t([\pi])\leq n-3$ for all $\pi \in S_\mathrm{poly}(\Z_n)$. Moreover, our results in this section show that permutation polynomials will not yield good lower bound examples when $n$ has many prime factors.

For $n$ prime, each permutation $\pi\in S(\Z_n)$ admits a permutation polynomial: 
\[
f(x)=\sum_a \pi(a)\left(1-(x-a)^{n-1}\right).
\]
However, when $n$ is not prime, this is no longer the case. Note that we can restrict ourselves to polynomials of degree $d\leq \varphi(n)-1$ for $\varphi$ the Euler's totient-function.

Recall that for $\pi \in S(\Z_n)$, we define $\cyc(\pi)$ as the number of cycles of $\pi$ in its cycle decomposition. For integers $n\geq 2$, we define
\[
c_{\textup{poly}}(n):=\min_{\pi\in S_{\textup{poly}}(\Z_n)} \max_{k\in [n]}\cyc(x\mapsto \pi(x+k)).
\]

The following lemma will follow easily from the Chinese remainder theorem and proves the lower bound $c_{\textup{poly}}(n)\geq 2^k$ when $n$ has $k$ different prime divisors.
\begin{lem}
\label{lem:product_cpoly}
Let $n$ and $m$ be coprime integers. Then $c_{\textup{poly}}(nm)\geq c_{\textup{poly}}(n)c_{\textup{poly}}(m)$.
\end{lem}
\begin{proof}
By the Chinese remainder theorem, the map
\[
h:\Z_{mn}\to \Z_{n}\times \Z_m,\quad x\mapsto (x\bmod n,x\bmod m)
\]
is a ring isomorphism. Let $f\in S_{\textup{poly}}(\Z_{nm})$ and let $a_0,\dots,a_d\in \Z_{nm}$ such that
\[
\pi(x)=a_0+a_1x+\dots+a_dx^d, \text{ for all }x\in \Z_{nm}.
\]
We define $f_1(y)\in \Z_{n}[y]$ as the permutation polynomial
\[
f_1(y):=
(a_0 \bmod n)+(a_1\bmod n) y+\dots+ (a_d\bmod n)y^d.
\]
and similarly $f_2(y)\in \Z_{m}[y]$ as 
\[
f_2(y):=
(a_0 \bmod m)+(a_1\bmod m) y+\dots+ (a_d\bmod m)y^d.
\]
With $h$ the ring isomorphism from the Chinese remainder theorem,
for each $x\in \Z_{nm}$
\[
h(f(x))=(f_1(x\bmod n),f_2(x\bmod m)).
\]
This gives $f$ a `product structure' and in fact that is the key property that we will need for the remainder of the argument. 

By definition, there exist $i_1\in \Z_n$ such that $s_1:x\mapsto f_1(x+i_1)$ has $n_0\geq c_{\text{poly}}(n)$ cycles and $i_2\in \Z_m$ such that $s_2:x\mapsto f_2(x+i_2)$ has $m_0\geq c_{\text{poly}}(m)$ cycles. Let $i=h^{-1}(i_1,i_2)\in \Z_{nm}$. We show that $x\mapsto f(x+i)$ has at least $n_0m_0\geq c_{\text{poly}}(n)c_{\text{poly}}(m)$ cycles. 
Note that  \[
h(f(x+i))=(f_1,f_2)(h(x+i))=(s_1(x\bmod n),s_2(x\bmod m)).
\]
In particular, for any two sets $A_1\subseteq \Z_n$ and $A_2\subseteq \Z_m$ with $s_1(A_1)=A_1$ and $s_2(A_2)=A_2$, we find that $(s_1,s_2)(A_1\times A_2)= A_1\times A_2$ and so $x\mapsto f(x+i)$ has at least one cycle fully contained in $h^{-1}(A_1\times A_2)$. This provides at least $n_0m_0$ cycles of $x\mapsto f(x+i)$. 
\end{proof}
By the lemma above, we are left in the case in which $n$ has only one type of prime divisor, that is, $n=p^k$ for $p$ prime and $k\geq 2$ an integer. In this case, rather than having a direct product structure, it turns out all permutation polynomials have the wreath product structure studied in Section~\ref{sec:semi-direct} as shown in the proof below.
\begin{lem}
\label{lem:semi_product_cpoly}
    Let $m,n\geq 2$ be integers. Then $c_{\textup{poly}}(mn)\geq c_{\textup{poly}}(n)+1$.
\end{lem}
\begin{proof}
For $x\in \Z_{nm}$, we denote by $R(x)\in\{0,\ldots, m-1\}$ and $Q(x)\in \{0,\ldots, n-1\}$ the unique numbers such that $x=R(x)+mQ(x) \bmod{mn}$. This gives a bijection 
\[
h:\Z_{mn}\to \Z_m\times \Z_{n}, \quad x\mapsto (R(x) \bmod{m},Q(x) \bmod{n}).
\]
For any $\sigma\in S(\Z_{mn})$, the permutation $h\circ\sigma \circ h^{-1}\in S(\Z_m\times \Z_n)$ has the same cycle type as $\sigma$. Moreover, the cyclic shift $(x\mapsto x+1)\in S(\Z_{mn})$ corresponds to the map $c\in G\subseteq S(\Z_m\times \Z_n)$ from~(\ref{eq:c_semidirect}). That is, $c(x,y)=h(h^{-1}(x,y)+1)$.

Let $f=a_d z^d+\dots+a_1z+a_0\in \Z_{mn}[z]$ be a permutation polynomial. Let $\sigma\in S(\Z_m\times \Z_n)$ be the induced map $\sigma(x,y)=h(f(h^{-1}(x,y)))$. Note that for every integer $k$ the permutation induced by $f(z)+k$ and the permutation $\sigma\circ c^k$ have the same cycle type. It therefore suffices to show that that there is an integer $k$ such that the permutation $\sigma\circ c^k$ has at least $c_{\textup{poly}}(n)+1$ cycles. To do this, we will first show that $\sigma$ is of the form $\sigma(x,y)=(\pi(x),\pi_x(y))$ as in (\ref{eq:semidirect}).

For $r\in \{0,\ldots, m-1\}$ and $q\in \{0,\ldots,n-1\}$ we have 
\begin{align*}
f(r+qm)
&=\sum_{i=0}^d a_i(qm+r)^i\\
&=\sum_{i=0}^d a_i\sum_{j=0}^i\binom{i}{j}r^{i-j}(qm)^j\\
&=\sum_{i=0}^d a_ir^i+m\left(\sum_{j=1}^d q^j m^{j-1}\sum_{i=j}^d\binom{i}{j}r^{i-j}a_i\right)\\
&=f(r)+m\left(\sum_{j=1}^d q^j m^{j-1}\sum_{i=j}^d\binom{i}{j}r^{i-j}a_i\right)\\
&=R(f(r))+m\left(Q(f(r))+\sum_{j=1}^d q^j m^{j-1}\sum_{i=j}^d\binom{i}{j}r^{i-j}a_i\right).
\end{align*}
It follows that for $(r,q)\in \Z_m\times \Z_n$ we have 
\[
\sigma(r,q)=(\pi(r), \pi_r(q))
\]
where $\pi\in \Z_m[x]$ is obtained from $f$ by reducing the coefficients modulo $m$, and $\pi_r\in \Z_n[y]$ is obtained from
\[
g_r(y):=Q(f(r))+\sum_{j=1}^d y^j m^{j-1}\sum_{i=j}^d\binom{i}{j}r^{i-j}a_i\in \Z_{mn}[y]\]
by reducing the coefficients modulo $n$.
Since $\sigma$ is surjective, so are the functions induced by $\pi$ and the $\pi_r$. It follows that $\pi$ and the $\pi_r$ are permutation polynomials.  

Let $t\in \Z_n$ be such that $y\mapsto \pi_0(y+t)$ has at least $c_{\textup{poly}}(n)$ cycles. Then by Lemma~\ref{lem:semi_direct_limitation_3cycles} there is a $k$ so that $\sigma\circ c^k$ has at least $c_{\textup{poly}}(n)+1$ cycles, completing the proof.
\end{proof}
In the proof above, we used the fact that each permutation polynomial over $\Z_{nm}$ can be written in the `wreath product form', which has been shown in more generality in \cite[Lemma 2.3]{gorcsos2018permutation}. 

We remark that permutation polynomials and their properties (including cycle structure) have been widely studied, especially with coefficients in a finite field (see e.g. the surveys~\cite{Hou15survey,Wang2025survey} for references). However, we were unable to find works which considered the cycle structure of the permutation \emph{and its cyclic shifts}, though connections with the `wreath product form' have been made previously for permutation polynomials (over finite fields), for example by 
Bors and Wang~\cite{BorsWang2022cosetwiseaffinefunctionswreath}.

The following is a corollary of Lemma~\ref{lem:semi_product_cpoly} and the fact that each permutation has a cyclic shift with at least $2$ cycles.
\begin{cor}[Restatement of Proposition~\ref{prop:permpoly}]
If $n$ is composite, then
    $c_{\textup{poly}}(n)\geq 3$.
\end{cor}
Moreover, the lemmas allow us to determine that permutations given by affine functions are the most difficult to cyclically sort, amongst the permutation polynomials, for certain values of $n$.
\begin{cor}
If $p\geq 3$ is prime and $k\geq 1$ an integer, then
    $c_{\textup{poly}}(p^k)=k+1$ and
    $c_{\textup{poly}}(2p^k)=2(k+1)$.
\end{cor}
\begin{proof}
    The first lower bound follows from Lemma~\ref{lem:semi_product_cpoly} and the second uses the first lower bound and Lemma~\ref{lem:product_cpoly}. The upper bounds follow from the lower bounds on $t(p^k)$ and $t(2p^k)$ proved by \cite{adin2025circularsorting} which used affine functions (that is, permutation polynomials of degree 1).
\end{proof}
Finally, we remark that the lemmas in this section can be used to show that a certain avenue towards improving the lower bounds on $t(n)$ computationally is futile. As a simple corollary of the two lemmas, we find that $c_{\textup{poly}}(24)\geq c_{\text{poly}}(3)c_{\text{poly}}(8)\geq 2\cdot (2+1+1)=8$. This means that we cannot find a better lower bound example for $t(24)$ using permutation polynomials compared to our example in Section~\ref{sec:computational}, for example.

\section{Computational results}
\label{sec:computational}
In this section, we describe the results that we obtained via computer experimentation as well as our approach in obtaining them. In the auxiliary data to the arXiv submission of this paper, we included an example of a permutation achieving the lower bound from Table~\ref{tab:summary_tn} for every composite $n \leq 44$. This also contains the code for our computational results such as $t(25)=22$, Proposition~\ref{prop:quadratic} and Proposition~\ref{prop:avoid_small_cycles}.

\subsection{Computational approach for \texorpdfstring{$t(25)\leq 22$}{t25leq22}}
\label{subsec:t25}
Note that $n=25$ is the smallest composite number for which $t(n)\leq n-3$ is not guaranteed by Theorem \ref{thm:div_by_2_or_3}. We show by computer search that $t(25)\leq 22$. 

Naively iterating over all permutations would be too computationally expensive.
The connection to orthomorphisms and strong complete mappings from Lemma~\ref{lem:orthomorphism_doublefixed} and Lemma~\ref{lem:completemapping_transposition} aids us greatly here in our computer search: we may restrict our search for a permutation $\pi\in S_n$ with $t([\pi])=n-2$ to permutations that are also a strong complete mapping, of which there are much fewer~\cite{oeisA071607}. For example, for $n=25$ there are $24!>6\times 10^{23}$ permutations of $\Z_{25}$ that fix the element $0$, but only $78309000$ of these are strong complete mappings (also known as the number of ways of arranging $25$ non-attacking queens on a $25\times 25$ toroidal chessboard~\cite{oeisA007705}). 

We use a backtracking approach for generating strong complete mappings. We generate a function $f:\{0,1,\dots,24\}\to \{0,1,\dots,24\}$ by setting $f(0)=0$ and then for $x=1,2,\dots,24$ we choose a value for $f(x)$ to ensure we keep the `strong complete mapping' property: the values $\{f(y) : y\in \{0,1,\dots,x\}\}$ are all distinct, the values $\{f(y)-y : y\in \{0,1,\dots,x\}\}$ are all distinct and the values $\{f(y)+y : y\in \{0,1,\dots,x\}\}$ are all distinct.  Once we have generated a strong complete mapping, we can check the cyclic structure of the map seen as a permutation and those of its cyclic shifts. If we find at least $3$ cycles in one of the cyclic shifts or we cannot assign the next value, then we backtrack and try the next available value for $f(x)$. It is possible to already `prune' parts of the search tree by checking whether the function defined so far (or any of its cyclic shifts) already has any cycle of length between $2$ and $23$: if this is the case, it can never be extended to a permutation with cycle type $(1,24)$ in all its cyclic shifts.

This general approach can be sped up further as follows.
We may assume that $\pi(0)=0$.
For a permutation $\pi \in S_n$, we define $(\pi(x)-\pi(y))/(x-y)$ for $x,y\in \Z_n$ with $(x-y)$ invertible modulo $n$ as a \emph{slope} of $\pi$. 
Since $\pi$ is a bijection, the slope $0$ does not appear.
We put a total order on the elements of $\Z_n$ inherited from the integer representatives in $\{0,1,\dots,n-1\}$. This allows us to consider the smallest slope $c$ of $\pi$. Suppose that this is achieved by $(x_0,y_0)$, that is, $c=(\pi(x_0)-\pi(y_0))/(x_0-y_0)$.
We define $\pi':\Z_n\to\Z_n$ as 
\[
\pi': x\mapsto a^{-1}(\pi(ax+b)-\pi(b))
\]
where $a=x_0-y_0$ and $b=y_0$ are chosen so that
\begin{align*}
  \pi'(1)&=  a^{-1}(\pi(a+b)-\pi(b))=\frac{\pi(a+b)-\pi(b)}{(a+b-b)}=(\pi(x_0)-\pi(y_0))/(x_0-y_0)=c,\\
  \pi'(0)&=  a^{-1}(\pi(b)-\pi(b))=0.
\end{align*}
Since $x_0,y_0$ are distinct, $a\neq 0$.
For $a\in (\Z_n)^\times$ and $b\in \Z_n$ the map
\[
\pi \mapsto \pi': x\mapsto a^{-1}\pi(ax+b)-\pi(b)
\]
preserves the set of slopes, the strong-complete-mapping property, and $C(\pi)$ (the set of cycle types of the cyclic shifts of $\pi$). Therefore, we only have to show that $t([\pi])\leq 22$ for all strong complete mappings $\pi:\Z_{25}\to \Z_{25}$ for which the smallest slope is achieved for $(x,y)=(0,1)$. 
This means that when we try the next value for $\pi(x)$, we can compute the slopes it will create and only have to try the options for which each new slope is at least $\pi(1)-\pi(0)=\pi(1)$.

\subsection{Avoiding small cycles}
\label{sec:avoiding_small_cycles}
Let $n\geq 2$ be an integer.
For integers $a=a(n)\geq 2$, the following claims are equivalent.
\begin{itemize}
    \item For each $\pi\in S_n$, there is a cyclic shift of $\pi$ which has two fixed points or a cycle of length at most $a$ that is not a fixed point.
    \item For each orthomorphism $\pi\in S_n$, there is a cyclic shift of $\pi$ which has a cycle of length at most $a$ that is not a fixed point.
    \item For each $\pi\in S_n$, there is a subset $A\subseteq \Z_n$ with $2\leq |A|\leq a$ such that $\pi(A)=A+i$ for some $i\in \Z_n$.
\end{itemize}
The equivalent claims are vacuously true for $a(n)=n-1$ and are true for $a(n)=n-2$ if and only if there is no permutation $\pi\in S_n$ with cycle type $(1,n-1)$ in all cyclic shifts.

Let $a^*(n)$ denote the smallest integer $a\geq 1$ for which the first claim holds.
Conjecture~\ref{conj:non-prime} is equivalent to $a^*(n)\leq n-2$ for all composite $n$. 
In Section~\ref{sec:non_prime}, we show that $a^*(n)=1$ when $n\geq 3$ is even and $a^*(n)\leq 2$ when $n\geq 4$ is divisible by $3$.
This gives hope that perhaps a stronger upper bound can be proved on $a^*(n)$ for composite $n$, for example $a^*(n)\leq p-1$ when $p$ is a prime divisor of $n$. Unfortunately, such hopes are diminished due to the following result.

\begin{prop}
\label{prop:avoid_small_cycles}
  For all $\pi \in S_{25}$, there is a subset $A$ with $2\leq |A|\leq 7$ for which $\pi(A)=A+i$ for some $i\in \Z_{25}$. Moreover, there are orthomorphisms  $\pi\in S_{25}$ that avoid cycles of lengths $2,3,4,5,6$ in all cyclic shifts.  That is, $a^*(25)=7$.
\end{prop}
We prove this proposition using a computer search similar to the one described in the previous subsection.
If we view $\pi$ and $\pi'$ as the same if there is $a\in\Z_n^\times$ and $b,b'\in \Z_n$ for which $\pi':x\mapsto a^{-1}\pi(a(x+b))+b'$, then there are only two `different' examples which avoid cycles of length at most 6. The lexicographically smallest representatives are
\[
(0, 2, 4, 23, 14, 18, 10, 22, 16, 12, 1, 8, 5, 9, 20, 24, 21, 3, 17, 13, 7, 19, 11, 15, 6).
\]
and
\[
(0, 2, 12, 1, 11, 17, 3, 16, 4, 15, 21, 6, 20, 5, 7, 18, 10, 19, 23, 8, 24, 9, 13, 22, 14).
\]
Both have the following cycle types (in some order) for the cyclic shifts: $20$ times the cycle type $(1,24)$, $4$ times $(1,7,7,10)$ and $1$ time $(1,8,8,8)$.

\subsection{Values of \texorpdfstring{$t(n)$}{t(n)} for \texorpdfstring{$n$}{n} non-prime}
\label{subsec:table_of_values}

In this subsection we summarize the result of our theorems and computer searches for $n \in [44]$. In particular, we give the exact value of $t(n)$ for composite $n \leq 21$ as well as an example of a permutation that achieves this value (see Table \ref{tab:summary_tn}). 

For all composite $n\leq 44$, except $n = 25$ and $n= 35$, the number $n$ is divisible by $2$ or by $3$ and therefore $t(n)\leq n-3$ by Theorem~\ref{thm:div_by_2_or_3}. In the special case $n=25$, a more advanced computer search allowed us to show that $t(n)\leq n-3$ as well in Section~\ref{subsec:t25}. Our computer searches are constructive, thus for all lower bounds given on $t(n)$, we also have an associated permutation achieving it. 

We remark that, already for $n=8$, Lemma \ref{lem:semi_product_cpoly} implies that any permutation showcasing $t(8) \geq 5$, such as the one presented in Table \ref{tab:summary_tn}, is not a permutation polynomial over $\Z_8$. More generally, most of the lower bound examples are not permutation polynomials. 

\begin{table}[!ht]
\centering
\begin{tabular}{c|c}
\toprule
$n$ & $t(n)$ bounds \\
\midrule
4  & 1 \\
6  & 3 \\
8  & 5 \\
9  & 6 \\
10 & 7 \\
12 & 9 \\
14 & 11 \\
15 & 12 \\
16 & 13 \\
18 & 15 \\
20 & 17 \\
21 & 18 \\
22 & [18,19] \\
24 & [20,21] \\
\bottomrule
\end{tabular}
\hspace{3cm}
\begin{tabular}{c |c}
\toprule
$n$ & $t(n)$ bounds \\
\midrule
25 & 22 \\
26 & [22,23] \\
27 & 24 \\
28 & [24,25] \\
30 & [26,27] \\
32 & [28,29] \\
33 & 30 \\
34 & [30,31] \\
35 & [31,33] \\
36 & [31,33] \\
38 & [34,35] \\
39 & 36   \\
40 & [35,37]  \\
42 & [37,39]  \\
\bottomrule
\end{tabular}

\vspace{1cm}

{\renewcommand{\arraystretch}{1.2}
\begin{tabular}{c|l}
\toprule
 $n$ & Lower bound permutation $\pi = (\pi(0), \pi(1),\ldots,\pi(n-1))$ \\
\midrule
4   & (0,1,3,2) \\
6   & (0,1,3,5,2,4) \\
8   & (0,1,3,5,2,7,4,6) \\
9   & (0,1,3,4,6,2,8,5,7) \\
10  & (0,1,4,9,3,8,7,2,6,5) \\
12  & (0,4,7,9,6,1,2,11,8,10,5,3)\\
14  & (0,5,8,13,9,7,11,4,1,12,3,10,6,2) \\
15  & (0,11,9,5,7,14,3,6,13,12,10,2,8,4,1) \\
16  & (0,2,4,1,12,13,9,14,10,5,3,6,11,7,8,15) \\
18  & (0,5,1,17,7,11,13,10,9,4,2,16,3,12,15,8,6,14) \\
20  & (0,15,7,10,2,9,17,3,14,6,12,19,5,8,4,13,16,11,1,18) \\
21  & (0,2,4,13,19,16,17,6,8,11,9,7,5,18,12,1,20,3,14,10,15) \\
\end{tabular}
}

\caption{The two top tables summarise the best lower and upper bounds we obtained on $t(n)$ for composite $n\leq 42$. We note that $t(n)=n-2$ for $n$ prime. In the bottom table, an example of a permutation achieving the lower bound is presented for composite $n \leq 21$.}
\label{tab:summary_tn}
\end{table}

\newpage

\section{Conclusion}\label{sec:conclusion}
Via the connection to strong complete mappings, we proved Conjecture~\ref{conj:non-prime} for $n$ divisible by $2$ or~$3$ and disproved Conjecture~\ref{conj:affine_prime}. 
Using quadratic orthomorphisms we obtained counterexamples to Conjecture~\ref{conj:affine_prime} for many primes, but we did not manage to extend this to an infinite family.
\begin{problem}
    Is there an infinite family of primes $p$ for which there exist $a\neq b\in \Z_p$ such that \[
    \pi_{a,b,p}(x):=[a,b](x)=((a-b)/2)x^{(p+1)/2}+((a+b)/2)x
    \]
    satisfies $t([\pi_{a,b,p}])=p-2$?
\end{problem}
Perhaps it is even possible to characterise all primes $p$ and $a, b \in \Z_p$ such that the map $[a, b](x) + c$ is a permutation of cycle type $(1, p - 1)$ for all $c \in \Z_p$. 

\smallskip

New ideas seem to be needed to prove Conjecture~\ref{conj:non-prime} for other values of $n$. We suggest the following special case of this.
\begin{problem}
    Is $t(p^2)=p^2-3$ for all primes $p$? 
\end{problem}
The lower bound in the problem above can be obtained using affine permutations.

We still have only limited insights into how to prove upper bounds on $t(n)$ and in particular leave the following problem open. 
\begin{problem}
    Is there a value of $n$ for which $t(n)\leq n-4$? 
\end{problem}
Despite being unable to exclude the possibility that $t(n)\geq n-3$ for all $n$, the best lower bound which holds for all $n$ remains the lower bound $t(n)\geq n - e(\ln n + 1)$  from~\cite{adin2025circularsorting} obtained by probabilistic arguments. We leave open the problem of finding explicit constructions of $\pi$ with $n-t([\pi])=O(\log n)$ for all values of $n$.
In particular, our lower bounds from Theorem~\ref{thm:lower_bounds} behave poorly for $n=pqr$ with $p,q,r= \Theta(n^{1/3})$, since it only tells us that $n-t(n)=O(n^{1/3})$. 
Similarly, affine functions perform poorly when $n$ has many prime factors. For example, there are values of $n$ with  $k=\Theta(\log n/\log\log n)$ different prime factors and for those value of $n$, Lemma~\ref{lem:product_cpoly} implies that $n-t([\pi])\geq 2^k=n^{\Omega(1/\log\log n)}$ for all affine $\pi$ (more generally, all permutation polynomials).

\medskip

We would also like to highlight the following question~\cite[Question 3.9]{adin2025circularsorting}: is it true that $t(n)\geq t(n-1)$ for all values of $n$? In particular, it would be interesting to compute the values of $t(24)$ and $t(30)$ since these are the first places at which monotonicity could still fail:
\begin{itemize}
    \item $20\leq t(24)\leq 21$ with $t(23)=21$,
    \item $26\leq t(30)\leq 27$ with $t(29)=27$.
\end{itemize}

From a sorting perspective, it seems natural that adding an element to the domain cannot make it easier to sort. However, this does not hold at the level of individual permutations. For example, 
\begin{align*}
  &\pi_{6,19}=(0, 6, 12, 18, 5, 11, 17, 4, 10, 16, 3, 9, 15, 2, 8, 14, 1, 7, 13)\\
  &\sigma_{6,19}=(0, 6, 12, 18, 5, 11, 17, 4, 10, 16, 3, 9, 15, 2, 8, 14, 1, 7, 13,19)
\end{align*}
satisfy $t([\pi_{6,19}])=19-3=16$ and $t([\sigma_{6,19}])=20-5=15$
(here  $\pi_{6,19}(x)=6x\bmod 19$). 

We remark that the following different notion of monotonicity might also be interesting.
\begin{conj}
\label{conj:our_multiplicity}
Let $c(n)=n-t(n)$. Is it true that $c(nm)\geq c(n)$ for all integers $n,m\geq 2$?    
\end{conj}
We proved that when restricting to permutation polynomials, we even find the inequality $c_{\text{poly}}(nm)\geq c_{\text{poly}}(n)+1$, but this is not true in general since for example $t(8)=5$ (by exhaustive search).  If Conjecture~\ref{conj:our_multiplicity} is true, then Conjecture~\ref{conj:non-prime} would be reduced to the case $n=pq$ with $p,q$ odd primes. 

\bibliographystyle{abbrv}
\bibliography{ref.bib}

\appendix

\section{Existence of strong complete mappings}
The results in this section are known, but we provide them for completeness since we did not find a reference for the second lemma which was easily accessible and in English.
\begin{lem}
\label{lem:conj_for_evenn}
There exists an orthomorphism $f:\Z_n\to\Z_n$ if and only if $n$ is odd.
\end{lem}
\begin{proof}
    Suppose that $n$ is even and that $f:\Z_n\to\Z_n$ is a bijection. Then $f(0)+f(1)+\dots+f(n-1)=0+1+\dots+n-1=n(n-1)/2$. For $n$ even, $n(n-1)/2$ will be non-zero modulo $n$.
    For $f(x)-x$ however, the sum will be $0$. This means $g(x)=f(x)+x$ cannot be a bijection when $f$ is a bijection and $n$ is even. 

    When $n=1$, there is only one bijection $f$ and $g(x)=f(x)+x$ is the same bijection.
    For $n$ odd and $n\geq 3$, there exists $a\in \Z_n$ such that $\gcd(a,n)=1$ and $\gcd(a+1,n)=1$. Now $f(x)=ax\mod n$ is a bijection and also $g(x)=f(x)+x=ax+x=(a+1)x$ is a bijection.
\end{proof}
\begin{lem}
    There is a bijection $f:\Z_n\to\Z_n$ such that $g(x)=f(x)+x$ and $h(x)=f(x)-x$ are both also bijections if and only if $n$ is odd and not divisible by 3.
\end{lem}
\begin{proof}
    When $n$ is not divisible by 2 or 3, we can find $a$ with $\gcd(a,n)=\gcd(a-1,n)=\gcd(a+1,n)=1$. For example, write $n=2k+1$ and let $a=k$. Now $x\mapsto ax$ satisfies the requirements.

    For the other direction, the case $2\mid n$ is handled by the previous lemma so suppose that $3\mid n$ and $n$ is odd. Then
    \[
    1^2+2^2+\dots+n^2=n(n+1)(2n+1)/6
    \]
    is not divisible by $n$. Indeed, $3\mid n$ and so 3 does not divide $n+1$ or $2n+1$. So $3$ divides the number on the right-hand side in the displayed equation one fewer time than it divides $n$.

    Next, we show that the existence of $f$ implies that the sum of the squares is 0 mod $n$, i.e. divisible by $n$. Since $f,g,h$ are all bijections,
    \[
    1^2+2^2+\dots+n^2=\sum_{x\in \Z_n} g(x)^2=\sum_{x\in \Z_n} (f(x)+x)^2=\sum_{x\in \Z_n} f(x)^2+\sum_{x\in \Z_n} x^2+2\sum_{x\in \Z_n} xf(x)
    \]
    \[
    1^2+2^2+\dots+n^2=\sum_{x\in \Z_n} h(x)^2=\sum_{x\in \Z_n} (f(x)-x)^2=\sum_{x\in \Z_n} f(x)^2+\sum_{x\in \Z_n} x^2-2\sum_{x\in \Z_n} xf(x).
    \]
    Combined, this shows $2\sum_{x\in \Z_n} xf(x)=0$.
    But since $f$ is a bijection, we also have
    \[
    1^2+2^2+\dots+n^2=\sum_{x\in \Z_n} f(x)^2.
    \]
    So
    \[
    1^2+2^2+\dots+n^2=\sum_{x\in \Z_n} g(x)^2=\sum_{x\in \Z_n} f(x)^2+\sum_{x\in \Z_n} x^2=2(1^2+\dots+n^2).
    \]
    Since $n$ is odd, $2$ is invertible and so this equation implies that the sum of squares is zero, as desired.
\end{proof}

\section{Limitations of the wreath product construction}\label{sec:B}
We show that when $\sigma\in S(\Z_{mn})$ takes the wreath product form defined in Section~\ref{sec:semi-direct}, then some cyclic shift $\sigma c^k$ of $\sigma$ has at least $4$ cycles whenever ($m=2$ and $n\geq 3$) or ($n=2$ and $m\geq 3$).

The case in which $m=2$ is relatively easy to analyse.
\begin{prop}
    Let $n\geq 3$ be an integer.
    Suppose that $\sigma(x,y)=(\pi(x),\pi_x(y))$ for $\pi\in S(\Z_2)$ and $\pi_{0},\pi_{1}\in S(\Z_n)$. Then there is a cyclic shift of $\sigma$ with at least $4$ cycles.
\end{prop}
\begin{proof}
By replacing $\sigma$ by $\sigma c$ if necessary, we may assume that $\pi$ is the identity permutation. 

For every $t\in \{0,1,\dots,n-1\}$, we have $\sigma c^{2t}(x,y)=(x,\pi_x(y+t))$. If $\pi_x$ has cycle type $(1,n-1)$ in all cyclic shifts for all $x\in \Z_2$, then $\sigma$ itself has at least $2+2=4$ cycles. If not, let $x\in \Z_2$ and $t\in \{0,1,\dots,n-1\}$ be such that $y\mapsto \pi_x(y+t)$ has at least three cycles. Then $\sigma c^{2t}$ has at least $3$ cycles on $\{x\}\times \Z_n$ so at least $3+1=4$ cycles in total.
\end{proof}

When $n=2$, the analysis becomes more complicated and we first need the following lemma.

\begin{lem}\label{lem:complete}
Let $\pi\in S(\Z_n)$ be a complete mapping. Viewing $\pi$ as a permutation of $\{0,\ldots, n-1\}$, we have
\[
|\{i\in \{0,\ldots, n-1\}: i+\pi(i)\geq n\}|=(n-1)/2.
\]
\end{lem}
\begin{proof}
Define $f:\{0,\ldots, n-1\}\to \{0,1\}$ by setting $f(i)=1$ if $i+\pi(i)\geq n$, and $f(i)=0$ otherwise. Since $\pi$ is a complete mapping, $i\mapsto i+\pi(i)-nf(i)$ is a permutation of $\{0,\ldots, n-1\}$. Using that $0+1+\cdots+(n-1)=n(n-1)/2$, it follows that 
\[
n(n-1)=\sum_{i=0}^{n-1} (i+\pi(i))=\sum_{i=0}^{n-1}(i+\pi(i)-nf(i))+n\sum_{i=0}^{n-1} f(i)=n(n-1)/2+ n\sum_{i=0}^{n-1} f(i).
\]
This shows that $\sum_{i=0}^{n-1}f(i)=(n-1)/2$.
\end{proof}

\begin{prop}
    Let $n=2$ and $m\geq 3$. Then $\sigma(x,y)=(\pi(x),\pi_x(y))$ satisfies $t([\sigma])\leq 2m-4$ for every choice $\pi\in S(\Z_m)$ and $\{\pi_x\}_x\subseteq S(\Z_2)$.
\end{prop}
\begin{proof}
If $\pi$ does not have cycle type $(1,m-1)$ in all cyclic shifts, then some cyclic shift of $\pi$ has a fixed point~$d$ and at least two other cycles. Now $\sigma$ will have a cyclic shift with fixed points $(d,0)$ and $(d,1)$ and at least two more cycles. So we may assume that $\pi$ has cycle type $(1,m-1)$ in all cyclic shifts. In particular, $m$ is odd. 

For $r\in \Z_m$, denote by $d_r$ the unique fixed point of the cyclic shift $x\mapsto \pi(x+r)$. So $\pi(d_r+r)=d_r$. Since every $x\in \Z_m$ is the fixed point of some cyclic shift of $\pi$, we have $\{d_0,\ldots, d_{m-1}\}=\Z_m$ and $\{d_0+0,\ldots, d_{m-1}+(n-1)\}=\pi^{-1}(\Z_m)=\Z_m$. So the map 
\[
\tau:\Z_m\to \Z_m, \quad r\mapsto d_r
\]
is a complete mapping.   

There exist $b_x\in \Z_2$ such that $\pi_x(y)=y+b_x$ for all $x\in \Z_m$. Recall that given an integer $k=r+tm$ where $r\in \{0,\ldots, m-1\}$ and $t\in \Z$, we define 
\[
w_k:\Z_m\to \Z_2,\qquad 
w_k(x)=\begin{cases}t&\text{if } x\in\{0,1,\ldots, m-1-r\},\\
t+1&\text{otherwise}.
\end{cases}
\]
Let $r\in \{0,\ldots, m-1\}$. We choose $t_r\in \{0,1\}$ such that $w_{r+mt_r}(d_r)+b_{d_r+r}=0$.
By Lemma~\ref{lem:sigmarotate} and our definition of $\{b_x\}$,
\begin{equation}
\label{eq:shift_sigma_2p}
\sigma c^{r+mt_r}(x,y)=(\pi(x+r), \pi_{x+r}(y+w_{r+mt_r}(x)))=(\pi(x+r), y+w_{r+mt_r}(x)+b_{x+r}).
\end{equation}

With $x=d_r$, our choice of $t_r$ ensures that 
$\sigma c^{r+mt_r}$ has fixed points $(d_r,0)$ and $(d_r,1)$. 
Suppose that the remaining cycle of $\pi$ is $(x_1~x_2~\dots~x_{m-1})$. (Here we omitted the dependence of $x_i$ on $r$ for readability purposes.)
As $\sigma c^{r+mt_r}$ already has two fixed points, it has only three cycles if and only if $(\sigma c^{r+mt_r})^{m-1}(x_1,0)=(x_1,1)$ by Lemma~\ref{lem:lift}.
By (\ref{eq:shift_sigma_2p}), this is equivalent to
\[\sum_{i=1}^{m-1} (b_{x_i+r} +w_{r+mt_r}(x_i))=1.\]
By definition, $w_{r+mt_r}(x)=w_r(x)+t_r$ for each $x$. As $m-1\equiv 0\pmod{2}$, we find that
\[
\sum_{i=1}^{m-1} (b_{x_i+r} +w_{r+t_rm}(x_i))=\sum_{i=1}^{m-1} (b_{x_i+r} +w_{r}(x_i))=b_{d_r+r}+w_r(d_r)+\sum_{x\in \Z_m} (b_{x+r}+w_{r}(x)).
\]
Using that $\sum_{x\in \Z_m}w_r(x)=r$, 
\[
b_{d_r+r}+w_r(d_r)+\sum_{x\in \Z_m} (b_{x+r}+w_{r}(x))=
r+(b_{d_r+r}+w_r(d_r))+\sum_{x\in \Z_m} b_x.
\] 

We conclude that $(\sigma c^{r+mt_r})^{m-1}$ has at most three cycles if and only if 
\begin{equation}\label{eq:threecycles}
(r+w_r(d_r))+b_{d_r+r}=1+\sum_{x\in \Z_m} b_x.
\end{equation}

Now suppose for contradiction that $(\sigma c^{r+mt_r})^{m-1}$ has at most three cycles for every $r\in \{0,\ldots, m-1\}$. By summing all $m$ equations (\ref{eq:threecycles}) we obtain 
\[
\sum_{r=0}^{m-1}(r+w_r(d_r))+\sum_{r=0}^{m-1}b_{d_r+r}=\sum_{r=0}^{m-1}(1+\sum_{x\in \Z_m} b_x).
\]
Using that $m$ is odd and that $\{d_r+r:r\in \{0,1,\dots,m-1\}\}=\Z_m$, we get
\[
\sum_{r=0}^{m-1}(r+w_{r}(d_r))\equiv 1\pmod{2}.
\]
However, since $\tau:r\mapsto d_r$ is a complete mapping, we can apply  Lemma~\ref{lem:complete} to obtain
\[
\sum_{r=0}^{m-1}(r+w_{r}(d_r))=r(r-1)/2+\sum_{r=0}^{m-1}w_{r}(\tau(r))=r(r-1)/2+(r-1)/2\equiv 0\pmod{2}.
\]
This contradiction shows that the assumption that $\sigma c^{r+mt_r}$ has at most $3$ cycles for every $r$ is false, completing the proof.
\end{proof}

\end{document}